\documentclass{amsart}
\usepackage[T1]{fontenc}
\usepackage[utf8]{inputenc}
\usepackage{amsmath,amsthm,amsfonts,amscd,amssymb,eucal,latexsym,mathrsfs}
\usepackage{stmaryrd}
\usepackage{enumerate}
\usepackage{hyperref}
\usepackage[all]{xy}
\usepackage{etoolbox}

\usepackage{tikz,tikz-cd}
\usetikzlibrary{shapes.geometric}
\usetikzlibrary{arrows}

\usetikzlibrary{positioning}
\usepackage{float}
\usepackage{MnSymbol}
\usetikzlibrary{matrix}

\newcommand{\toc}{\tableofcontents}

\theoremstyle{plain}
\newtheorem{theorem}{Theorem}[section]
\newtheorem*{theorem*}{Theorem}

\newtheorem{corollary}[theorem]{Corollary}
\newtheorem*{corollary*}{Corollary}
\newtheorem{cor}[theorem]{Corollary}
\newtheorem{lemma}[theorem]{Lemma}

\newtheorem{proposition}[theorem]{Proposition}

\theoremstyle{definition}
\newtheorem{remark}[theorem]{Remark}

\newtheorem{definition}[theorem]{Definition}
\newtheorem*{definition*}{Definition}

\usetikzlibrary{calc,graphs}
\usepackage{xcolor}
\usetikzlibrary{arrows,decorations.pathmorphing}

\newcommand{\e}{\varepsilon}

\newcommand{\IR}{\mathbb{R}}

\newcommand{\CI}{\mathbb{C}}

\renewcommand{\Re}{\operatorname{Re}}

\DeclareMathOperator{\ssi}{\Leftrightarrow}
\DeclareMathOperator{\impl}{\Rightarrow}

\newcommand{\conv}{\operatorname{conv}}

\newcommand{\ts}{\textsection}

\newcommand{\ip}[1]{\langle#1\rangle} 
\newcommand{\norm}[1]{\|#1\|} 
\newcommand{\abs}[1]{|#1|}

\newcommand{\St}{\mathop{\mathrm{St}}}

\newcommand{\hh}{\mathcal{H}}

\title{Separation theorems for bounded convex sets of bounded operators}

\author{Mika\"el Pichot}
\author{Erik S\'eguin}
\address{Mika\"el Pichot, McGill University, 805 Sherbrooke St W., Montr\'eal, QC H3A 0B9, Canada}\email{mikael.pichot@mcgill.ca}
\address{Erik S\'eguin, University of Waterloo,
200 University Avenue West,
Waterloo, Ontario, N2L 3G1, Canada}\email{e2seguin@uwaterloo.ca}

\begin{document}

\maketitle

\begin{abstract}  
We establish new metric characterizations for the norm (respectively, ultraweak) closure of the convex hull of a bounded set in an arbitrary $C^*$-algebra (respectively, von Neumann algebra), and provide applications of these results to the majorization theory.  
\end{abstract}

\section{Introduction}

Let $M$ be a $\sigma$-finite von Neumann algebra, let $x\in M$, and let $X\subset M$ be a bounded set. 
 We prove:

\begin{theorem}\label{T - intro} The following  conditions are equivalent:
\begin{enumerate}
\item $x\in \overline{\conv X}$, where $\conv X$ denotes the convex hull of $X$ and the closure is with respect to the ultraweak topology;
\item for every  $y\in M$, there exists $z\in X
$ such that 
\[
\|x-z\|_2\leq \|y-z\|_2.
\]
\end{enumerate}
Here by $\|\cdot\|_2$ we mean the 2-norm associated with a faithful normal state. 
\end{theorem}

This result provides a metric characterization for the ultraweakly closed convex hull of a bounded set in $M$. It can be viewed as a separation theorem for bounded convex sets and can be used in some cases as a substitute to the Hahn-Banach theorem. In the present paper, we are mainly interested in applications to the majorization theory. In the most classical case of matrix majorization, we prove the following. 
 
\begin{corollary}\label{C - matrix}  Let $A$ and $B$ be arbitrary $n\times n$ complex matrices. The following conditions are equivalent:
\begin{enumerate}
\item there exists unitary matrices $U_1,\ldots ,U_n$ and positive numbers $t_1, \ldots, t_n$ such that 
\[
\sum_{i=1}^n t_i=1,\qquad A=\sum_{i=1}^n t_i U_iBU_i^* ;
\]
\item
for every complex matrix $C$, there exists a unitary matrix $U$ such that 
\[
\|A-UBU^*\|_2\leq \|C-UBU^*\|_2.
\]
\end{enumerate} 
Here by $\|\cdot\|_2$ we mean the Frobenius norm.
\end{corollary}

We refer to \cite{Ando}, \cite{Ando2}, and \cite{Marshall} for  introductions to the majorization theory (see also \ts\ref{S - Majorization finite} below for the basic definitions). When $A$ and $B$ are Hermitian matrices, Conditions (1) and (2) are equivalent to the well-known majorization relation $A\prec B$ between  self-adjoint matrices (see Def. \ref{D - majorization}).  More general majorization results, for arbitrary elements in von Neumann algebras and $C^*$-algebras, will be established in this paper.

The proof of Theorem \ref{T - intro} relies on a separation lemma for bounded sets in Hilbert spaces (Lemma \ref{L - hyperplanes}).  Although we are mainly discussing majorization  in the present paper, we mention that in \cite{tarskistable}, a related separation lemma (for general, not necessarily bounded sets) was used as an intermediate step in the proof of an Ulam-type stability characterization of amenability for groups, in terms of positive definite maps with values in $B(\hh)$.  

In \ts\ref{S - maj vN} we generalize Theorem \ref{T - intro} to arbitrary von Neumann algebras, and use this generalization to prove (in \ts\ref{S - maj C*}) the following separation result for the norm closed convex hull of a bounded set in an arbitrary $C^*$-algebra.  

\begin{theorem}\label{T - majo cs} 
Let $A$ be a unital $C^*$-algebra, ${X\subset A}$ be a bounded set, and ${x\in A}$. The following are equivalent: 
\begin{enumerate} 
\item ${x\in\overline{\conv X}}$, where the closure is with respect to the norm topology;
\item for every state ${\psi\in\St(A)}$ and ${y\in A}$, there exists ${z\in X}$ such that 
\[\|{x-z}\|_{2,\psi}\leq\|{y-z}\|_{2,\psi}.\]
\end{enumerate} 
\end{theorem} 

The majorization theory in  von Neumann algebras (see \cite{Kamei,Kamei2,Kamei3,Hiai,HiaiNakamura0, HiaiNakamura}) and $C^*$-algebras  (see \cite{S,NS,majocs,MS})  is well-studied.
 For self-adjoint elements in a $C^*$-algebra, Theorem 1.1 in \cite{majocs}
provides a general majorization criterion in terms of lower semi-continuous traces. At the cost of testing against all states on $A$, we obtain the following characterization of the norm closure of the convex hull of the unitary orbit of a non-necessarily self-adjoint operator in a general $C^{*}$-algebra. 

\begin{corollary}\label{T - majo cs intro}
Let $A$ be a unital $C^*$-algebra and let $x,y \in A$. The following are equivalent:
\begin{enumerate}
\item $x\in \overline{\conv\{ uyu^*: u\in U(A)\}}$, where $U(A)$ denotes the unitary group of $A$ and the closure is with respect to the norm topology; 
\item for every  state $\psi\in\St(A)$ and $z\in A$, there exists $u\in U(A)
$ such that 
\[
\|x-uyu^*\|_{2,\psi}\leq \|z-uyu^*\|_{2,\psi}.
\]
\end{enumerate}
\end{corollary}

We conclude this paper with some remarks on submajorization (\ts\ref{S - Remarks}).

\toc

\section{A lemma}

Let $\hh$ be a (real or complex) Hilbert space. 
We  begin with a separation lemma which provides a metric characterization of the norm closure of a bounded convex set in $\hh$.  We do not assume that $\hh$ is separable.

\begin{lemma}\label{L - hyperplanes} Let  $X\subset \hh$ be a bounded set, let $\xi\in \hh$ be a vector, and let $\Omega\subset\hh$ be a dense set.   The following are equivalent: 
\begin{enumerate}
\item $\xi\in\overline {\conv X}$, where the closure is with respect to the norm topology;
\item for every $\eta\in\Omega$, there exists $\zeta\in X$ such that
\[
\|\xi-\zeta\|\leq\|\eta-\zeta\|.
\]
\end{enumerate}
\end{lemma}

\begin{proof} Lemma 2.2 in \cite{tarskistable} shows that (1) implies (2); we establish  the 
converse here. Suppose that ${\xi\notin {\overline {\conv X}} }$ and (2) holds.  

Let $\eta_0$ denote the orthogonal projection of $\xi$ onto ${\overline {\conv X}} $. Since $\Omega$ is dense, there exists a sequence  $(\eta_n)\subset \Omega$ such that $\eta_n\to \eta_0$.
 We may assume that  $\eta_n\neq \xi$ for every $n$. In this case, the median (real) hyperplane $M_n$ for $[\xi,\eta_n]$ (thus, $M_n$ consists of all vectors equidistant to $\xi$ and $\eta_n$) divides $\hh$ into two disjoint open half-spaces. We call $H_n$ the open half-space associated with $M_n$ which contains $\eta_n$.

Since $\eta_0$ is the orthogonal projection of $\xi$ onto ${\overline {\conv X}} $, the latter set is included in $H_0$. We claim that if ${\overline {\conv X}} $ is a bounded set, it must be included in $H_n$ for every $n$ sufficiently large. 

Note that there exists a ball $B_0$ of center $[\xi,\eta_0]\cap M_0$ in $M_0$ such that ${\overline {\conv X}} $ is included in $B_0\times \ell$, where $\ell$ denotes the span of $\xi-\eta_0$. For every $\e>0$ and every $n$ sufficiently large, $M_n\cap B_0\times \ell$ is included in the $\e$-neighbourhood of $B_0$. We let $\e<\frac 1 4 \|\xi-\eta_0\|$, and find an $N$ such that for every $n\geq N$, ${\overline {\conv X}} \subset H_n$.  

In particular, there exists $n\geq 0$ such that 
\[
\norm{\eta_n-\zeta}<\norm{\xi-\zeta}
\]
 for every $\zeta\in {\overline {\conv X}} $.
 
Since $\eta_n\in \Omega$, we obtain for ${\zeta\in X}$ as in (2)  
\[
\norm{\xi-\zeta}\leq\norm{\eta_n-\zeta}< \norm{\xi-\zeta}
\]
which is a contradiction.
\end{proof}

\begin{remark}
It is clear from the proof that the second condition may be replaced with
\begin{enumerate}
\item[(2')] the closure of the set 
\[
\Omega_\xi:=\{\eta\in\Omega :\ \exists \zeta\in X : \norm{\xi-\zeta}\leq\norm{\eta-\zeta}\}
\]
contains the boundary of $\overline{\conv X}$.
\end{enumerate}
The equivalence between (1) and (2) under the condition that $\Omega=\hh$ in (2) is a particular case of  \cite[Lemma 2.2]{tarskistable}.  In particular, it follows, in the setting of Lemma \ref{L - hyperplanes},  that $\Omega_\xi$ is dense in $\hh$ if and only if $\Omega_\xi = \Omega$.
\end{remark}

We next show that the assumption that $X$ is bounded in Lemma \ref{L - hyperplanes} is  important. The following counterexample shows the result fails in general for unbounded convex sets. 

\begin{proposition} The following conditions are equivalent:
\begin{enumerate}
\item there exists a non-empty closed convex set $X\subset \hh$, a dense set $\Omega\subset \hh$, and an element $\xi\in  \hh$ for which the equivalence in Lemma \ref{L - hyperplanes} fails;
\item  $\dim_\IR \hh \geq 2$.
\end{enumerate}
\end{proposition}

\begin{proof}
We first prove no counterexample to Lemma \ref{L - hyperplanes} can exist in dimension one. Suppose $X$ is a non-empty subset of $\IR=B(\IR)$. Then $\conv X$ is an interval in $\IR$, and for every dense set $\Omega\subset \IR$,  $\xi \in {\overline {\conv X}} $ if and only if   $\forall\eta\in\Omega$, $\exists \zeta\in X$ such that
\[
|\xi-\zeta|\leq |\eta-\zeta|.
\]
Indeed, if $\xi\not\in {\overline {\conv X}} $, then by density we may choose an $\eta$ which lies in the non-empty open interval between $\xi$ and ${\overline {\conv X}} $. Then it is clear that for every $\zeta\in X$, $\eta$ is closer to $\zeta$ than $\xi$ is. Conversely, if $\xi\in {\overline {\conv X}} $ and $\eta$ is an element in $\Omega$, then regardless of the relative position of $\xi$ and $\eta$, it is not difficult to prove that there will always exist some $\zeta$ in $X$ which is closer (although not necessarily strictly) to $\xi$ than to $\eta$.

Suppose now that $\hh$ is a  Hilbert space of finite real dimension $n\geq 2$, which we shall identify with $\IR^n$ for convenience. 
We consider the closed convex set 
\[
X=\{(x_1,\ldots, x_n)\in \hh : x_1\leq 0\}
\] 
 let $\xi=(1,0,\ldots, 0)$, and let $\Omega:=\hh\setminus \IR\xi$. Clearly $\xi\not\in X$ and $\Omega$ is a dense set. We claim that Condition (2) in Lemma \ref{L - hyperplanes} holds. Indeed, let $\eta\in \Omega$. Then, since $\eta\not\in \IR\xi$, the median (real) hyperplane for $[\eta,\xi]$ intersects the boundary hyperplane 
 \[
 \{(x_1,\ldots, x_n)\in \hh : x_1= 0\}.
 \] 
 Therefore, there exists an element $\zeta\in X$ such that 
\[
\|\xi-\zeta\|<\|\eta-\zeta\|.
\] 
This proves that (2) holds.

A similar argument clearly works in infinite dimension.
\end{proof}

\section{Separation theorems and majorization}\label{S - Majorization finite}

In this section we let $M$ be a $\sigma$-finite von Neumann algebra and fix  a faithful normal state  $\psi$ on $M$. We write $\|x\|_2:=\sqrt{\psi(x^*x)}$ for the 2-norm associated with the state $\psi$, and view $M$ as a dense subset of $L^2(M,\psi)$. 
A direct application of Lemma \ref{L - hyperplanes} gives:

\begin{theorem}[Separation theorem in $\sigma$-finite von Neumann algebras]\label{T - Separation sigma finite}
 Let $X\subset M$ be a bounded set and ${x\in M}$. The following are equivalent:
\begin{enumerate}
\item $x\in \overline{\conv X}$, where the closure is with respect to the ultraweak topology;
\item for every $y\in M$, there exists $z\in X
$ such that 
\[
\|x-z\|_2\leq \|y-z\|_2.
\]
\end{enumerate}
\end{theorem}

\begin{proof}    Since $X$ is bounded and $M$ is dense in the $2$-norm topology, it follows by Lemma \ref{L - hyperplanes} that
\[
x\in \overline {\conv X}
\]
if and only if for every $y\in M$, there exists $z\in X$ such that
\[
\|x-z\|_2\leq \|y-z\|_2
\]
where the closure is relative to the $2$-norm. Since the $2$-norm topology coincides with the ultrastrong topology on bounded sets, this concludes the proof. 
\end{proof}

\begin{remark} It is not difficult to extend Theorem \ref{T - Separation sigma finite} to direct sums of $\sigma$-finite algebras. Namely, let $S$ be a set, let $M=\bigoplus_{\alpha\in S} M_\alpha$ be a direct sum of $\sigma$-finite von Neumann algebras, and for every 
$\alpha\in S$, let $\psi_\alpha$ be a faithful normal state on $M_\alpha$. If $X\subset M$ is a bounded set and $x \in M$, then the following are equivalent:
\begin{enumerate}
\item $x\in \overline{\conv\, X}$, where the closure is with respect to the ultraweak topology;
\item for every $\alpha_1,\ldots, \alpha_n\in S$ and $y\in M$, there exists $z\in X
$ such that  
\[
\sum_{i=1}^n\|x-z\|_{2, \psi_{\alpha_i}}^2\leq \sum_{i=1}^n\|y-z\|_{2, \psi_{\alpha_i}}^2.
\]
\end{enumerate}
We shall in fact establish a completely general statement, valid for all von Neumann algebras, in \ts\ref{S - maj vN} below. The particular case of direct sums of $\sigma$-finite algebras is interesting on its own, as it is sufficient for some applications, including to double duals of separable $C^*$-algebras (the fact that double duals of separable $C^*$-algebras are direct sums of $\sigma$-finite algebras is well-known to experts---for example, Elliott uses it in \cite[Lemma 3.4]{Elliott}; we refer to \cite{Sherman} for more details). 
\end{remark}

As mentioned in the introduction, Theorem \ref{T - Separation sigma finite} has applications in the majorization theory.
We first recall the following definition (see \cite{Kamei,Kamei2,Kamei3,Hiai,HiaiNakamura0,HiaiNakamura}). 

\begin{definition}\label{D - majorization}
Suppose $x,y\in M$ are self-adjoint, and $\psi$ is tracial and factorial. We say that \emph{$x$ is majorized by $y$}, and write $x\prec y$, if the following two conditions hold: 
\begin{itemize}
\item[(a)] $\psi(x)=\psi(y)$ 
\item[(b)] $\psi((x-r)_+)\leq \psi((y-r)_+)$ for all $r\in 
\IR$.
\end{itemize}
\end{definition} 

Here $(x-r)_+$ denotes the element
obtained from $x$ by functional calculus with the function $x \mapsto (x - r)_+ := \max(x - r, 0)$. 

 If $M$ is a finite factor, $x,y\in M$ are self-adjoint elements, and $\psi$ is tracial, then Condition (1) in Theorem \ref{T - HK} is  well-known to be equivalent to $x\prec y$,
by the results of Hiai and Nakamura:

\begin{theorem}[see Theorem 2.1 in \cite{HiaiNakamura}]\label{T - HK}
Suppose $x,y\in M$ are self-adjoint, and $\psi$ is tracial and factorial.  The following are equivalent:
\begin{enumerate}
\item $x\in \overline{\conv\{ uyu^*: u\in U(M)\}}$, where the closure is with respect to the ultraweak topology;
\item $x\prec y$.
\end{enumerate}
\end{theorem}

Our next result provides a metric characterization of the ultraweak closure of the convex hull  of the unitary orbit of an arbitrary operator in an arbitrary $\sigma$-finite von Neumann algebra $M$. (In fact, the assumption that $M$ is $\sigma$-finite can be dropped, see \ts\ref{S - maj vN}.)  Condition (2) can be interpreted as a  majorization condition  $x\prec y$ for arbitrary elements  $x,y$ in $M$. 

\begin{theorem}\label{T - Majorization finite factors}
 Let $x,y \in M$. The following are equivalent:
\begin{enumerate}
\item $x\in \overline{\conv\{ uyu^*: u\in U(M)\}}$, where the closure is with respect to the ultraweak topology;
\item for every  $z\in M$, there exists $u\in U(M)
$ such that 
\[
\|x-uyu^*\|_2\leq \|z-uyu^*\|_2.
\]
\end{enumerate}
\end{theorem}

In the finite dimensional case, Theorem \ref{T - Majorization finite factors} provides a tracial majorization criterion  for  general (not necessarily Hermitian) $n\times n$  real or complex matrices (Cor.\ \ref{C - matrix}). To the best of our knowledge, this criterion is new even for matrix algebras.

\begin{proof}[Proof of Theorem \ref{T - Majorization finite factors}]
Consider the bounded set
\[
X = \{uyu^*:u\in U(M)\}
\] 
It follows by Theorem \ref{T - Separation sigma finite} that
\[
x\in \overline {\conv\{uyu^* : u\in U(M)\}}
\]
if and only if for every $z\in M$, there exists $u\in U(M)$ such that
\[
\|x-uyu^*\|_2\leq \|z-uyu^*\|_2.\qedhere
\]
\end{proof}

We remark that a direct application of Lemma 2.2 in \cite{tarskistable} (as opposed to Lemma \ref{L - hyperplanes} in the present paper) would only establish an equivalence between the following two conditions:
\begin{enumerate}
\item $x\in \overline{\conv\{ uyu^*: u\in U(M)\}}$, where the closure is with respect to the ultraweak topology; and,
\item[(2')] for every  $z\in L^2(M,\psi)$, there exists $u\in U(M)
$ such that 
\[
\|x-uyu^*\|_2\leq \|z-uyu^*\|_2.
\]
\end{enumerate}
This equivalence already implies the proposed characterization for matrices, since  $L^2(M,\psi)=M$ if $M=M_n(\CI)$. 

It also follows that (2) in Theorem \ref{T - Majorization finite factors} is equivalent to (2') for every $\sigma$-finite von Neumann algebra.

On the other hand, the full generality of Lemma \ref{L - hyperplanes} is not required to establish
Theorem \ref{T - Majorization finite factors}.   This is because a $\sigma$-finite von Neumann algebra $M$ is ``projection closed''  in the following sense. 

\begin{definition}
Let $\hh$ be a Hilbert space. A pair $X\subset Y$ of subsets of $\hh$ is \emph{projection closed} if for every $y\in Y$, the orthogonal projection of $y$ onto the closure of the convex hull of $X$ belongs to $Y$.
\end{definition}

For example, if $Y$ is a finite dimensional subspace of $\hh$, then the pair ${X\subset Y}$ is projection closed for every subset ${X\subset Y}$. Furthermore, it is easy to check that Lemma \ref{L - hyperplanes} may be replaced by the following result (setting $Y=M\subset \hh=L^2(M,\psi)$), which is a direct modification of Lemma 2.2 in \cite{tarskistable}, in the proof of Theorem \ref{T - Separation sigma finite}.

\begin{lemma} Let ${X\subset Y}$ be a projection closed pair in $\hh$ and let ${\xi\in Y}$ be a vector. The following are equivalent: 
\begin{enumerate}
\item $\xi\in \overline{\conv X}$, where the closure is with respect to the norm topology; 
\item for every $\eta\in Y$, there exists $\zeta\in X$ such that 
\[\|\xi-\zeta\|\leq\|\eta-\zeta\|.\]
\end{enumerate}
\end{lemma}

\begin{proof}
The proof is identical to that of Lemma 2.2 in \cite{tarskistable}.
\end{proof}

Finally, we mention that for self-adjoint elements, Theorem \ref{T - Majorization finite factors} can be combined with the Hiai-Nakamura theorem  \cite[Theorem 6.4]{HiaiNakamura} that the ultraweak closure and the norm closure of the convex hulls of unitary orbits must coincide. This gives the following result.

\begin{theorem}\label{T - Majorization finite factors Hiai Nakamura}
	If $x,y \in M$ are self-adjoint elements, where $M$ is a $\sigma$-finite von Neumann algebra, then the following are equivalent:
	\begin{enumerate}
		\item[(1')] $x\in \overline{\conv\{ uyu^*: u\in U(M)\}}$, where the closure is with respect to the  operator norm topology; 
		\item[(2)] for every  $z\in M$, there exists $u\in U(M)
		$ such that 
		\[
		\|x-uyu^*\|_2\leq \|z-uyu^*\|_2.
		\]
	\end{enumerate}
\end{theorem}

It seems to be an open question whether or not the ultraweak and operator norm closures of $\conv\{ uyu^*: u\in U(M)\}$ coincide in general (see \cite[p.\ 36]{HiaiNakamura}) for an arbitrary element $y$ in a $\sigma$-finite von Neumann algebra $M$.

\section{Separation theorem in general von Neumann algebras}\label{S - maj vN}

In this section we let $M$ be a general (not necessarily $\sigma$-finite) von Neumann algebra. If $A$ is a unital $C^*$-algebra and ${\psi\in \St(A)}$ is a state, then we let ${\ip{\cdot,\cdot}_{\psi}}$ denote the positive semi-definite sesquilinear form on $A$ defined by 
\[\ip{x,y}_{\psi}=\psi(y^*x)\]
and let ${\norm{\cdot}_{2,\psi}}$ denote the induced semi-norm. Furthermore, if ${F\subset\St(A)}$ is a finite set of states, then we let ${\ip{\cdot,\cdot}_F}$ denote the positive semi-definite sesquilinear form on $A$ defined by 
\[\ip{x,y}_F=\sum_{\psi\,\in\,F}\ip{x,y}_{\psi}\]
and let ${\norm{\cdot}_{2,F}}$ denote the induced semi-norm. 

\begin{theorem}[Separation theorem in general von Neumann algebras]\label{T - Separation general vN} 
Let ${X\subset M}$ be a bounded set and ${x\in M}$. The following are equivalent: 
\begin{enumerate} 
\item ${x\in\overline{\conv X}}$, where the closure is with respect to the ultraweak topology
\item for every normal state ${\psi\in\St_\sigma(M)}$ and ${y\in M}$, there exists ${z\in X}$ such that 
\[\|{x-z}\|_{2,\psi}\leq\|{y-z}\|_{2,\psi}\] 
\item there exists a separating set ${W\subset\St_\sigma(M)}$ of normal states such that for every finite subset ${F\subset W}$ and ${y\in M}$, there exists ${z\in X}$ such that 
\[\norm{x-z}_{2,F}\leq\norm{y-z}_{2,F}\]
\end{enumerate} 
\end{theorem} 

\begin{proof} 
Suppose first that ${(2)}$ holds and let ${F\subset\St_\sigma(M)}$ be a finite set of normal states. Let ${\psi\in\St_\sigma(M)}$ be the normal state defined by 
\[\psi=\frac{1}{\abs{F}}\sum_{\varphi\,\in\,F}\varphi\]
Then ${\norm{\cdot}_{2,\psi}=\abs{F}^{-1/2}\,\norm{\cdot}_{2,F}}$, which  yields the implication ${(2)\Rightarrow(3)}$. Now suppose that ${(3)}$ holds. Let $\mathcal{F}$ be the collection of all finite subsets of $W$, let ${F\in\mathcal{F}}$ be a finite set of normal states, and let 
\[N_F=\{y\in M:\norm{y}_{2,F}=0\}\]
Let ${\hh_F}$ denote the completion of ${M/N_F}$ to a Hilbert space with respect to the inner product induced on the quotient space by ${\ip{\cdot,\cdot}_F}$. Let ${\pi_F:M\to\hh_F}$ be the canonical map; then ${(3)}$ implies that for every ${y\in M}$, there exists ${z\in X}$ such that 
\[\norm{\pi_F(x)-\pi_F(z)}\leq\norm{\pi_F(y)-\pi_F(z)}\]
As ${\pi_F(M)}$ is dense in ${\hh_F}$, it follows by Lemma \ref{L - hyperplanes} that 
\[\pi_F(x)\in\overline{\conv(\pi_F(X))}=\overline{\pi_F(\conv X)}\]
where the closure is with respect to the norm topology; thus for every ${F\in\mathcal{F}}$ and every ${\varepsilon>0}$ there exists ${z_{F,\varepsilon}\in\conv X}$ such that 
\[\norm{x-z_{F,\varepsilon}}_{2,F}=\norm{\pi_F(x)-\pi_F(z_{F,\varepsilon})}<\varepsilon\]
Define a direction on ${\mathcal{F}\times(0,\infty)}$ by letting ${(E,\delta)\leq(F,\varepsilon)}$ if ${E\subset F}$ and ${\varepsilon\leq\delta}$. Let ${\psi\in W}$ be a normal state and ${\varepsilon>0}$ be a positive real number; then 
\[(E,\delta)\geq(\{\psi\},\varepsilon)\implies\norm{x-z_{E,\delta}}_{2,\psi}\leq\norm{x-z_{E,\delta}}_{2,E}<\delta\leq\varepsilon\]
and so ${\norm{x-z_{E,\delta}}_{2,\psi}\to0}$ for all ${\psi\in W}$. As $W$ is separating, the topology induced by ${\{\norm{\cdot}_{2,\psi}:\psi\in W\}}$ agrees with the ultrastrong topology on bounded sets; it therefore follows that ${(z_{E,\delta})\to x}$ ultrastrongly, hence ultraweakly, which proves ${(3)\Rightarrow(1)}$. Now suppose that ${(1)}$ holds. Let ${\psi\in\St_\sigma(M)}$ be a normal state, let ${F=\{\psi\}}$, and let ${\hh_F}$ and ${\pi_F}$ be as above; then ${\pi_F(x)\in\overline{\conv(\pi_F(X))}}$, where the closure is with respect to the norm topology, and thus it follows by Lemma \ref{L - hyperplanes} that for every element ${y\in M}$, there exists ${z\in X}$ such that 
\[\|{x-z}\|_{2,\psi}=\norm{\pi_F(x)-\pi_F(z)}\leq\norm{\pi_F(y)-\pi_F(z)}=\|{y-z}\|_{2,\psi}\]
thereby proving the implication ${(1)\Rightarrow(2)}$. 
\end{proof}  


\begin{remark} 
It is natural to ask whether the conditions in Theorem \ref{T - Separation general vN} are equivalent to the following condition: 
\begin{enumerate} 
\setcounter{enumi}{3} 
\item there exists a separating set ${W\subset\St_\sigma(M)}$ of normal states such that for every ${\psi\in W}$ and ${y\in M}$, there exists ${z\in X}$ such that 
\[\|{x-z}\|_{2,\psi}\leq\|{y-z}\|_{2,\psi}\]
\end{enumerate} 
The following counterexample demonstrates that this is not so: let ${M=\mathbb{C}\oplus\mathbb{C}}$, let ${a=(1,0)}$, let ${b=(0,1)}$, let ${X=\{a,b\}}$, let ${x=(1,1)}$, and let ${\varphi,\psi\in\St_\sigma(M)}$ be the normal states defined by 
\[\varphi(y,z)=y,\qquad\qquad\psi(y,z)=z\]
Then ${W=\{\varphi,\psi\}}$ is a separating set of normal states and 
\[\norm{x-a}_{2,\varphi}=0,\qquad\qquad\norm{x-b}_{2,\psi}=0\]
which implies that ${(4)}$ holds. However, it is clear that ${x\notin\overline{\conv X}}$, and thus ${(1)}$ does not hold. 
\end{remark}

\section{Separation theorem in general $C^*$-algebras}\label{S - maj C*}

We use a standard double dual argument to deduce the following result from Theorem \ref{T - Separation general vN}.

\begin{theorem}\label{T - Separation cs alg} 
Let $A$ be a unital $C^*$-algebra, ${X\subset A}$ be a bounded set, and ${x\in A}$. The following are equivalent: 
\begin{enumerate} 
\item ${x\in\overline{\conv X}}$, where the closure is with respect to the norm topology
\item for every state ${\psi\in\St(A)}$ and ${y\in A}$, there exists ${z\in X}$ such that 
\[\|{x-z}\|_{2,\psi}\leq\|{y-z}\|_{2,\psi}\]
\end{enumerate} 
\end{theorem} 

\begin{proof} 
Suppose first that ${(2)}$ holds.  Let ${M=A^{**}}$, let ${\iota:A\hookrightarrow M}$ be the canonical embedding, let ${\psi\in\St_\sigma(M)}$ be a normal state, and let ${y\in M}$ be an arbitrary element; then Kaplansky's density theorem implies that there exists a bounded net ${(y_\alpha)_{\alpha\,\in\,\mathcal{I}}}$ in $A$ such that ${(\iota(y_\alpha))\to y}$ ultrastrongly. For every ${\alpha\,\in\,\mathcal{I}}$, let ${z_\alpha\in X}$ be an element such that ${\norm{x-z_\alpha}_{2,\psi}\leq\norm{y_\alpha-z_\alpha}_{2,\psi}}$. As ${\iota(X)}$ is bounded, it follows by ultraweak compactness that ${(z_\alpha)}$ admits a subnet ${(z_\beta)_{\beta\,\in\,\mathcal{J}}}$ such that ${(\iota(z_\beta))\to z}$ ultraweakly for some ${z\in\overline{\conv(\iota(X))}}$, where the closure is with respect to the ultraweak topology. Let ${\kappa=\sup\{\norm{\iota(z_\gamma)}_{2,\psi}:\gamma\in\mathcal{J}\}}$; then 
\begin{align*} 
\abs{\ip{\iota(x-y_\beta),\iota(z_\beta)}_\psi-\ip{\iota(x)-y,z}_\psi}&=\abs{\ip{\iota(x)-y,\iota(z_\beta)-z}_\psi-\ip{\iota(y_\beta)-y,\iota(z_\beta)}_\psi}
\\&\leq\abs{\ip{\iota(x)-y,\iota(z_\beta)-z}_\psi}+\kappa\,\norm{\iota(y_\beta)-y}_{2,\psi}\to0 
\end{align*} 
whence ${\Re\,\ip{\iota(x)-\iota(y_\beta),\iota(z_\beta)}_\psi\to\Re\,\ip{\iota(x)-y,z}_\psi}$. This implies that 
\[\Re\,\ip{\iota(x)-y,z}_\psi\geq\textstyle\frac{1}{2}\,(\norm{\iota(x)}_{2,\psi}^2-\norm{y}_{2,\psi}^2)\]
and therefore 
\[\norm{\iota(x)-z}_{2,\psi}\leq\|{y-z}\|_{2,\psi}\]
It then follows by Proposition \ref{T - Separation general vN} that ${\iota(x)\in\overline{\conv(\iota(X))}}$, where the closure is with respect to the ultraweak topology, and thus ${x\in\overline{\conv X}}$, where the closure is with respect to the weak topology; the implication ${(2)\Rightarrow(1)}$ then follows by Mazur's theorem. The implication ${(1)\Rightarrow(2)}$ follows immediately by passing to the double dual and applying Proposition \ref{T - Separation general vN}. 
\end{proof} 

We obtain the following majorization result as an immediate consequence of the above theorem. 

\begin{cor}
Let $A$ be a unital $C^*$-algebra and let $x,y \in A$. The following are equivalent:
\begin{enumerate}
\item $x\in \overline{\conv\{ uyu^*: u\in U(A)\}}$, where the closure is with respect to the norm topology;
\item for every  state $\psi\in\St(A)$ and $z\in A$, there exists $u\in U(A)
$ such that 
\[
\|x-uyu^*\|_{2,\psi}\leq \|z-uyu^*\|_{2,\psi}.
\]
\end{enumerate}
\end{cor}

\begin{proof}
It suffices to consider the bounded set 
\[X=\{uyu^*:u\in U(A)\}\]
and apply Theorem \ref{T - Separation cs alg}. 
\end{proof}

\section{Remarks on submajorization}\label{S - Remarks}

The following characterizations are variations on our previous results in the context of the submajorization theory (see \cite{Hiai,S,majocs}) which can be obtained by using the same techniques.

\begin{theorem}\label{T - Majorization finite factors 2}
Let $x,y \in M$. Consider the following conditions:
\begin{enumerate}
\item[(4)] $x\in \overline{\conv\{ uyu^*:  \|u\|\leq 1\}}$, where the closure is with respect to the ultraweak topology;
\item[(5)]for every normal state ${\psi\in\St_\sigma(M)}$ and every $z\in M$, there exists $u\in M
$ with $\|u\|\leq 1$ such that 
\[
\|x-uyu^*\|_{2,\psi}\leq \|z-uyu^*\|_{2,\psi};
\]
\item[(6)] $x\in \overline{\conv\{ uyv: u,v\in U(M)\}}$, where the closure is with respect to the ultraweak topology;
\item[(7)] for every normal state ${\psi\in\St_\sigma(M)}$ and every  $z\in M$, there exist $u,v\in U(M)
$ such that 
\[
\|x-uyv\|_{2,\psi}\leq \|z-uyv\|_{2,\psi}.
\]
\end{enumerate}
Then (4) $\ssi$ (5) $\impl$ (6) $\ssi$ (7).
\end{theorem}

\begin{proof}
The equivalences (4) $\ssi$ (5) and (6) $\ssi$ (7) are direct consequences  of  Theorem \ref{T - Separation general vN} and (4) $\impl$ (6) follows by the following remark: if  $M$ is a von Neumann algebra and $y\in M$, then
\[
\overline{\mathrm{conv}\{uyu^*: u\in U(M)\}}\subset  \overline{\mathrm{conv}\{uyu^*:\| u\|\leq1\}}\subset \overline{\mathrm{conv}\{uyv^*:u,v\in U(M)\}}
\]
where the closures are with respect to the ultraweak topology. Namely, it follows by the Kadison-Pedersen strengthening of the Russo-Dye theorem (see \cite[Theorem 1]{KadisonPerdersen}), that if $u\in M$ is an element such that $\|u\|<1-2n^{-1}$ for some integer $n$ greater than 2, then there exist $n$ unitary elements $u_1, \ldots, u_n$ in $M$ such that $u=n^{-1}(u_1+\cdots +u_n)$. This implies that:
\[
\mathrm{conv}\{uyu^*:\| u\|<1\}\subset \mathrm{conv}\{uyv^*:u,v\in U(M)\}.
\]
The inclusion $\overline{\mathrm{conv}\{uyu^*:\| u\|\leq1\}}\subset \overline{\mathrm{conv}\{uyv^*:u,v\in U(M)\}}$ follows. The first inclusion is obvious. 
\end{proof}

\begin{remark}
We note that the inclusions 
\[\overline{\mathrm{conv}\{uyu^*: u\in U(M)\}}\subset  \overline{\mathrm{conv}\{uyu^*:\| u\|\leq1\}}\subset \overline{\mathrm{conv}\{uyv^*:u,v\in U(M)\}}\]
are strict in general. Namely, if $x\in \overline{\mathrm{conv}\{uyu^*:\| u\|\leq1\}}$ and $0\leq t\leq 1$, then  $tx\in \overline{\mathrm{conv}\{uyu^*:\| u\|\leq1\}}$. Furthermore, when $y=1_M$,  the set $\overline{\mathrm{conv}\{uyv^*:u,v\in U(M)\}}$ is the closed unit ball of $M$, every element in $\overline{\mathrm{conv}\{uyu^*:\| u\|\leq1\}}$ is positive, and the set $\overline{\mathrm{conv}\{uyu^*: u\in U(M)\}}$ is reduced to ${\{1_M\}}$. In particular, the inclusions may be strict. 
\end{remark}

Analogous results also hold in the setting of  $C^*$-algebras by using the norm topology rather than the ultraweak topology.

\end{document}